\documentclass{aims} 
\usepackage{amsmath}
\usepackage{paralist}
\usepackage{amsmath,amsthm,amscd,amssymb,bm}
\usepackage{algpseudocode}
\usepackage{algorithm}
\usepackage[misc]{ifsym}
\usepackage{epsfig} 
\usepackage{epstopdf} 
\usepackage[colorlinks=true]{hyperref}
\hypersetup{urlcolor=blue, citecolor=red}
\allowdisplaybreaks

\def\currentvolume{X}
 \def\currentissue{X}
  \def\currentyear{200X}
   \def\currentmonth{XX}
    \def\ppages{X-XX}
     \def\DOI{10.3934/xx.xxxxxxx}

\newtheorem{theorem}{Theorem}[section]
\newtheorem{corollary}[theorem]{Corollary}

\newtheorem{lemma}[theorem]{Lemma}

\theoremstyle{definition}
\newtheorem{definition}[theorem]{Definition}

\newcommand{\cA}{\mathcal{A}}
\newcommand{\cB}{\mathcal{B}}
\newcommand{\cC}{\mathcal{C}}
\newcommand{\cD}{\mathcal{D}}

\newcommand{\cH}{\mathcal{H}}
\newcommand{\cI}{\mathcal{I}}

\newcommand{\cP}{\mathcal{P}}
\newcommand{\cQ}{\mathcal{Q}}
\newcommand{\cR}{\mathcal{R}}
\newcommand{\cS}{\mathcal{S}}

\newcommand{\cU}{\mathcal{U}}
\newcommand{\cV}{\mathcal{V}}

\newcommand{\cX}{\mathcal{X}}
\newcommand{\cY}{\mathcal{Y}}
\newcommand{\cZ}{\mathcal{Z}}

\newcommand{\vect}[1]{\mathrm{vec}(#1)}
\DeclareMathOperator*{\fold}{{fold}}
\DeclareMathOperator*{\unfold}{{unfold}}
\DeclareMathOperator*{\bcirc}{{bcirc}}
\DeclareMathOperator*{\bdiag}{bdiag}
\DeclareMathOperator{\rank}{rank}

\DeclareMathOperator*{\argmin}{argmin}
\newcommand{\st}{\mbox{subject to}}

\providecommand{\norm}[1]{\left\lVert#1\right\rVert}

\newcommand{\ve}{\mathbf{e}}

\newcommand{\vx}{\mathbf{x}}
\newcommand{\vy}{\mathbf{y}}
\newcommand{\vz}{\mathbf{z}}

\newcommand{\vo}{\mathbf{0}}

\title[Iterative Singular Tube Hard Thresholding Algorithms]{Iterative Singular Tube Hard Thresholding Algorithms for Tensor Recovery}

\author[R. Grotheer, S. Li, A. Ma, D. Needell and J. Qin]{}

\subjclass{Primary: 15A69; Secondary: 68W20, 15A83, 94A12, 65F55.}

\keywords{Tensor completion, tubal rank, hard thresholding, stochastic algorithm, image inpainting.}

\thanks{$^*$Corresponding author: Jing Qin}

\begin{document}
\maketitle

\centerline{\scshape
Rachel Grotheer$^{{\href{mailto:grotheerre@wofford.edu}{\textrm{\Letter}}}1}$,
Shuang Li$^{{\href{mailto:lishuang@iastate.edu}{\textrm{\Letter}}}2}$,
Anna Ma$^{{\href{mailto:anna.ma@uci.edu}{\textrm{\Letter}}}3}$,
Deanna Needell$^{{\href{mailto:deanna@math.ucla.edu}{\textrm{\Letter}}}4}$
and Jing Qin$^{{\href{mailto:jing.qin@uky.edu}{\textrm{\Letter}}}*5}$}

\medskip

{\footnotesize
 \centerline{$^1$Department of Mathematics, Wofford College, Spartanburg, SC 29303, USA}
} 

\medskip

{\footnotesize
 \centerline{$^2$Department of Electrical and Computer Engineering, Iowa State University, Ames, IA 50011, USA
}
}

\medskip

{\footnotesize
 \centerline{$^3$Department of Mathematics, University of California, Irvine, CA 92697, USA
}
}

\medskip

{\footnotesize
 \centerline{$^4$Department of Mathematics, University of California, Los Angeles, CA 90095, USA
}
}

\medskip

{\footnotesize
 \centerline{$^5$Department of Mathematics, University of Kentucky, Lexington, KY 40506, USA
}
}

\bigskip

 \centerline{(Communicated by Handling Editor)}


\begin{abstract}
Due to the explosive growth of large-scale data sets, tensors have been a vital tool to analyze and process high-dimensional data. Different from the matrix case, tensor decomposition has been defined in various formats, which can be further used to define the best low-rank approximation of a tensor to significantly reduce the dimensionality for signal compression and recovery. In this paper, we consider the low-rank tensor recovery problem when the tubal rank of the underlying tensor is given or estimated \emph{a priori}. We propose a novel class of iterative singular tube hard thresholding algorithms for tensor recovery based on the low-tubal-rank tensor approximation, including basic, accelerated deterministic and stochastic versions. Convergence guarantees are provided along with the special case when the measurements are linear. Numerical experiments on tensor compressive sensing and color image inpainting are conducted to demonstrate convergence and computational efficiency in practice.
\end{abstract}


\section{Introduction}
Due to the fast development of sensing and transmission technologies, data has been growing explosively, which makes the data representation a bottleneck in signal analysis and processing. Unlike the traditional vector/matrix representation, tensor provides a versatile tool to represent large-scale multi-way data sets. It has also been playing an important role in a wide spectrum of application areas, including video processing, signal processing, machine learning and neuroscience. However, when transitioning from two-way matrices to multi-way tensors, it becomes challenging to extend certain operators and preserve properties, e.g., tensor multiplication. Moreover, it is desirable to design fast numerical algorithms since computation involving tensors is highly demanding. To address these issues, decomposition of tensors becomes a feasible and popular approach to extract latent data structures and break the curse of dimensionality.

Many tensor decomposition methods have been developed by generalizing matrix decomposition. In particular, similar to the principal component decomposition for matrices, the CANDECOMP/PARAFAC (CP) decomposition \cite{cp1,cp2} decomposes a tensor as a sum of rank-one tensors. As a generalization of the matrix singular value decomposition (SVD), the Tucker form \cite{tucker} decomposes a tensor as a $k$-mode product of a core tensor and multiple matrices. Two popular Tucker decomposition methods have been proposed, including Higher Order Singular Value Decomposition (HOSVD) \cite{de2000multilinear} and Higher Order Orthogonal Iteration (HOOI) \cite{de2000best}. For the introduction of various tensor decomposition methods, we refer the readers to the comprehensive review paper \cite{kolda2009tensor} and references therein.

Finding the CP decomposition is an NP-hard problem and the low-CP-rank approximation is ill-posed \cite{hillar2013most}. Moreover, the Tucker decomposition is not unique as it depends on the representation order for each mode, which is also computationally expensive \cite{kolda2009tensor}. To make tensor decomposition more practically useful and possess uniqueness guarantees, a new type of tensor product, called t-product, and its corresponding related concepts including tubal rank, t-SVD, and tensor nuclear norm were developed for third-order tensors based on the fast Fourier transform \cite{kilmer2011factorization,kilmer2013third}. It has been applied to solve many high-dimensional data recovery problems, including tensor robust principal component analysis \cite{lu2016tensor}, tensor denoising and completion \cite{zhang2014novel,liu2016low,nimishakavi2018dual,wang2018noisy,wang2019noisy,zhang2019corrected}, imaging applications \cite{kilmer2011factorization,kilmer2013third}, image deblurring and video face recognition via order-$p$ tensor decomposition \cite{martin2013order}, and hyperspectral image restoration \cite{fan2017hyperspectral}. In this paper, we focus on third-order tensors for simplicity and adopt the t-product. All our results can be extended to higher order tensors using a well-defined tensor product following the analysis in \cite{martin2013order}.

Due to the compressibility of high-dimensional data, a large amount of data recovery problems can be considered as a tubal-rank restricted minimization problem whose objective function depends on the generation of measurements. For example, in tensor compressive sensing, measurements are generated by taking a single or a sequence of tensor inner products of the sensing tensors and the tensor to be restored. Accordingly, the objective function is usually in the form of the $\ell_2$-norm. The tensor restricted isometry property (RIP) and exact tensor recovery for the linear measurements is discussed in \cite{zhang2019tensor} as an extension of the matrix RIP.
Rank-restricted RIP in the matrix case \cite{foucart2019iterative} can be extended to the tubal-rank-restricted RIP \cite{zhang2019rip} in the tensor case, which will pave the theoretical foundation for our work. Motivated by the iterative singular tube thresholding (ISTT) algorithm \cite{wang2018noisy}, we propose iterative singular tube hard thresholding (ISTHT) algorithm which alternate gradient descent and singular tube hard thresholding (STHT). Note that ISTT uses the soft thresholding operator as the proximal operator of the tensor tubal nuclear norm, i.e., the solution to a convex relaxed tensor tubal nuclear norm minimization problem. Unlike ISTT, STHT serves as the proximal operator of the cardinality of the set consisting of nonzero tubes resulting from the tubal-rank constraint. In addition, STHT is based on the reduced t-SVD which provides a low tubal rank approximation of a given tensor \cite{kilmer2011factorization}. Iterative hard thresholding algorithms for several tensor decompositions, based on the low Tucker rank approximation of a tensor, are discussed in the compressive sensing scenario with linear measurements \cite{rauhut2017low}. Low-tubal-rank tensor recovery can also be solved by proximal gradient algorithm \cite{liu2023proximal}. To further speed up the computation, we also develop an accelerated version of ISTHT, which integrates the Nesterov scheme for updating the step size in an adaptive manner.

When the data size is extremely large, stochastic gradient descent can be embedded into the algorithm framework to reduce the computational cost. Recently, stochastic greedy algorithms (SGA) including stochastic iterative hard thresholding (StoIHT) and stochastic gradient matching pursuit (StoGradMP) have shown great potential in efficiently solving sparsity constrained optimization problems \cite{nguyen2017linear,qin2021stochastic} as well as in various applications \cite{durgin2019fast}. By iteratively seeking the support and running stochastic gradient descent, SGA can preserve the desired sparsity of iterates while decreasing the objective function value. Based on this observation, we develop stochastic ISTHT algorithms in non-batched and batched versions. Theoretical discussions have shown the proposed stochastic algorithm achieves at a linear convergence rate. Note that we develop iterative hard thresholding algorithms based on the low tubal rank approximation in a more general setting where the objective function is separable and satisfies the tubal-rank restricted strong smoothness and convexity properties. Numerical experiments on synthetic and real third-order tensorial data sets, including synthetic linear tensor measurements and RGB color images, have demonstrated that the proposed algorithms are effective in high-dimensional data recovery. It is worth noting that proposed tensor algorithms perform excellent especially for color image inpainting in terms of computational efficiency and recovery accuracy when the underlying image has a low tubal rank structure.

The rest of the paper is organized as follows. In Section~\ref{sec:pre}, fundamental concepts in tensor algebra are introduced, together with properties of tensor-variable functions, such as tensor restricted strong convexity and smoothness, and the tubal-rank restricted isometry property. In Section~\ref{sec:alg}, we describe the proposed non-stochastic/stochastic ISTHT algorithms in detail. Section~\ref{sec:cnvg} provides convergence analysis of our proposed algorithms for tensor completion with linear measurements as a special case.  Numerical experiments and corresponding performance results are reported in Section~\ref{sec:exp}. Finally, we draw a conclusion and point out future works in Section~\ref{sec:con}.

\section{Preliminaries}\label{sec:pre}
In this section, we provide preliminary knowledge for the best low-tubal-rank approximation and then define new concepts based on the tubal rank, including tubal-rank restricted {strong} convexity and smoothness of functions on a tensor space, and the tubal-rank restricted isometry property.

\subsection{Tensor Algebra}
To comply with traditional notation, we use boldface lower case letters to denote vectors by default, e.g., a vector $\vx\in\mathbb{R}^n$. Matrices are denoted by capital letters, e.g., $X\in\mathbb{R}^{n_1\times n_2}$  represents a matrix with $n_1$ rows and $n_2$ columns. Calligraphy letters such as $\mathcal{X}$ are used to denote tensors or unless specified otherwise (e.g., linear map), and let $\mathbb{R}^{n_1\times n_2\times n_3}$ be the space consisting of all real third-order tensors of size $n_1\times n_2\times n_3$. The set of integers $\{1,2,\ldots,n\}$ is denoted by $[n]$. Given a tensor, a fiber is a vector obtained by fixing two dimensions while a slice is a matrix obtained by fixing one dimension. If $\mathcal{X}\in\mathbb{R}^{n_1\times n_2\times n_3}$, then the fiber $\mathcal{X}(i,j,:)$ along the third dimension is a $n_3$-dimensional vector, which is also called the $(i,j)$-th tube or tubal-scalar of length $n_3$. Likewise, $\mathcal{X}(i,:,:)$, $\mathcal{X}(:,i,:)$ and $\mathcal{X}(:,:,i)$ are used to denote the respective horizontal, lateral and frontal slices. The $(i,j,k)$-th component of $\cX$ is denoted by $\cX_{i,j,k}$ or $\cX(i,j,k)$. A tensor $\cX$ is called f-diagonal if all frontal slices are diagonal matrices. For further notational convenience, the frontal slice $\mathcal{X}(:,:,i)$ is denoted by $X^{(i)}$. Using frontal slices, block diagonalization and circular operators which convert an $n_1\times n_2\times n_3$ tensor to a $(n_1n_3)\times(n_2n_3)$ matrix are defined as follows
\[
\bdiag(\cX)=\begin{bmatrix}
X^{(1)}&&&\\
&X^{(2)}&&\\
&&\ddots&\\
&&&X^{(n_3)}
\end{bmatrix}\hspace{-3pt},\,
\bcirc(\cX)=\begin{bmatrix}
X^{(1)}&X^{(n_3)}&\cdots&X^{(2)}\\
X^{(2)}&X^{(1)}&\cdots&X^{(3)}\\
&\cdots&\cdots&\\
X^{(n_3)}&X^{(n_3-1)}&\cdots&X^{(1)}
\end{bmatrix}\hspace{-4pt}.
\]
Without padding extra entries, another pair of operators are also defined to rewrite an $n_1\times n_2\times n_3$ tensor as an $(n_1n_3)\times n_2$ matrix and vice versa:
\[
\unfold(\cX)=\begin{bmatrix}X^{(1)}\\ \vdots\\ X^{(n_3)}\end{bmatrix},\quad
\fold(\unfold(\cX))=\cX.
\]

\begin{definition}\label{def:tprod}\cite{kilmer2013third}
Given two tensors $\mathcal{A}\in\mathbb{R}^{n_1\times n_2\times n_3}$ and $\mathcal{B}\in\mathbb{R}^{n_2\times n_4\times n_3}$, the tensor product (t-product) is defined as
\begin{equation}\label{eqn:tprod}
\mathcal{A}*\mathcal{B}=\mathrm{fold}(\mathrm{bcirc}(\mathcal{A})\cdot
\mathrm{unfold}(\mathcal{B})),
\end{equation}
where $\cdot$ is the standard matrix multiplication. We can also rewrite \eqref{eqn:tprod} as
\[
\mathrm{unfold}(\cA*\cB)=\mathrm{bcirc}(\cA)\mathrm{unfold}(\cB).
\]
By letting $\cC=\cA*\cB$, we get an equivalent form of the above equation \cite{chen2021regularized}
\[
\mathcal{C}(i,j,:)=\sum_{k=1}^{n_2}\mathcal{A}(i,k,:)\odot\mathcal{B}(k,j,:),\quad i=1,\ldots,n_1,\, j=1,\ldots,n_4,
\]
where $\odot$ is the circular convolution of two vectors by treating $1\times 1\times n_3$ tensors as vectors. Using the convolution-based formulation, we can get an intuitive connection between the $t$-product and other imaging applications, such as image deblurring \cite{chen2021regularized}.
\end{definition}
\noindent\textbf{Remark.} Using the definition of t-product and the block circulant operator, we can show that
\[
\bcirc(\cA*\cB)=\bcirc(\cA)\bcirc(\cB).
\]

Based on the definition of t-product, a lot of concepts in matrix algebra can be extended to the tensor case. For example, the transpose of $\cX$ is denoted by $\cX^T$, which is an $n_2\times n_1\times n_3$ tensor given by transposing each of the frontal slices and then reversing the order of transposed slices 2 through $n_3$, i.e., $(\cX)^T_{i,j,1}=\cX_{j,i,1}$ and $(\cX)^T_{i,j,k}=\cX_{j,i,n_3-k+1}$ for $i=1,\ldots,n_1$, $j=1,\ldots,n_2$ and $k=2,\ldots,n_3$. The identity tensor $\cI\in\mathbb{R}^{n_1\times n_1\times n_3}$ is a tensor whose first frontal slice is an $n_1\times n_1$ identity matrix and all other frontal slices are zero matrices. A tensor $\cQ\in\mathbb{R}^{n_1\times n_1\times n_3}$ is called orthogonal if $\cQ*\cQ^T=\cQ^T*\cQ=\cI$. By default, the space $\mathbb{R}^{n_1\times n_2\times n_3}$ is considered as a Hilbert space equipped with the inner product
\[
\langle \cX,\cY\rangle = \sum_{i,j,k}\cX_{i,j,k}\cY_{i,j,k},
\]
and the Frobenius norm given by
\[
\norm{\cX}=\sqrt{\sum_{i,j,k}\cX_{i,j,k}^2}.
\]
Note that if $\bm\theta$ is a linear map from a tensor space to a vector space, then $\norm{\bm\theta}$ stands for the operator norm unless otherwise stated.
Further detailed definitions and examples can be found in \cite{kilmer2013third}. Next, we introduce the singular value decomposition (t-SVD), based on the t-product.

\begin{definition}\label{def:tSVD}\cite{kilmer2013third}
Given $\cX\in\mathbb{R}^{n_1\times n_2\times n_3}$, there exist $\cU\in\mathbb{R}^{n_1\times n_1\times n_3}$, $\cS\in \mathbb{R}^{n_1\times n_2\times n_3}$ and $\cV\in \mathbb{R}^{n_2\times n_2\times n_3}$ such that
\begin{equation}\label{eqn:tSVD}
\cX=\cU * \cS * \cV^T.
\end{equation}
Here $\cS$ is f-diagonal and the number of nonzero tubes in $\cS$ is called the \textbf{tubal rank} of $\cX$, denoted by $\rank_t(\cX)$. \end{definition}

\noindent\textbf{Remark.} Due to the relationship between tensor and matrix, the t-SVD form yields the matrix form
\[
\unfold(\cX)=\bcirc(\cU)\bcirc(\cS)\unfold(\cV^T).
\]
Note that $\unfold(\cV^T)\neq(\unfold(\cV))^T$.

Moreover, it can be shown that the reduced t-SVD of $\cX$ denoted by
\begin{equation}
\cX_r={\sum_{i=1}^r\cU(:,i,:)*\cS(i,i,:)*\cV(:,i,:)^T}
\label{eqn:Xr}
\end{equation}
is a best tubal-rank-$r$ approximation of $\cX$ in the sense that \cite[Theorem 4.3]{kilmer2011factorization}
\begin{equation}
\cX_r=\argmin_{\cZ\in \mathbb{T}_r}\norm{\cX-\cZ}\label{eqn:cXr}
\end{equation}
where
\begin{equation}
\mathbb{T}_r=\{\cX_1*\cX_2\,\mid\,\cX_1\in\mathbb{R}^{n_1\times r\times n_3},\,
{\cX_2}\in\mathbb{R}^{r\times n_2\times n_3}\}
\label{eqn:Tr}
\end{equation}
is the set of all tubal-rank at most $r$ tensors of the size $n_1\times n_2\times n_3$. In addition, due to the block diagonalization of block-circulant matrices with circulant blocks under the Fourier transform, t-SVD can be efficiently obtained by using the matrix SVD and the fast Fourier transform \cite{kilmer2011factorization,zhang2014novel}. More specifically, letting $F_{n_3}\in\mathbb{C}^{n_3\times n_3}$ be the unitary discrete Fourier transform matrix, we have
\[
(F_{n_3}\otimes I_{n_1})\cdot \bcirc(\cX)\cdot (F_{n_3}^*\otimes I_{n_3})=\bdiag(\cD)
\]
where $\otimes$ is the Kronecker product of two matrices and
\[
\widehat{\cD}(:,:,k)=\widehat{\cU}(:,:,k)\widehat{\cS}(:,:,k)\widehat{\cV}(:,:,k)^T.
 \]
where $\widehat{\cX}$ is the Fourier transform of $\cX$ along the third dimension, i.e., $\widehat{\cX}(i,j,:)=\mbox{fft}(\cX(i,j,:))$.
To make the paper self-contained, we include the t-SVD algorithm in Algorithm~\ref{alg:tsvd}. The operators $\mbox{fft}(\cdot,[\,\,],3)$ and $\mbox{ifft}(\cdot,[\,\,],3)$ represent the Fourier transform along the third dimension. Note that the Fourier transform can be replaced by other unitary transforms, e.g., discrete cosine transform \cite{xu2019fast}.

\begin{algorithm}[H]
\caption{Tensor Singular Value Decomposition (t-SVD) {\cite{kilmer2011factorization}}}\label{alg:tsvd}
\begin{algorithmic}
\State\textbf{Input:} $\cX\in\mathbb{R}^{n_1\times n_2\times n_3}$.
\State\textbf{Output:} $\cU,\cS,\cV$.
\State $\widehat{\cX}=\mbox{fft}(\cX,[\,\,],3)$.
\For {$i=1,2,\ldots,n_3$}
\State Find the SVD of ${\widehat{\cX}(:,:,i)}$ such that $\widehat{\cU}(:,:,i)\widehat{\cS}(:,:,i)\widehat{\cV}(:,:,i)^T=\widehat{\cX}(:,:,i)$
\EndFor
\State $\cU=\mbox{ifft}(\widehat{\cU},[\,\,],3)$
\State $\cS=\mbox{ifft}(\widehat{\cS},[\,\,],3)$
\State $\cV=\mbox{ifft}(\widehat{\cV},[\,\,],3)$
\end{algorithmic}
\end{algorithm}

Using the t-SVD, we define the following singular tube hard thresholding operator.
\begin{definition}\label{def:STHT}
For $\cX\in\mathbb{R}^{n_1\times n_2\times n_3}$ and $r\leq \min\{n_1,n_2\}$, the singular tube hard thresholding (STHT) operator $\cH_r$ is defined as \begin{equation}\label{eqn:STHT}
\cH_r(\cX)=\cU*\cH_r(\cS)*\cV^T,\quad \cH_r(\cS)(i,i,:)=\vo,\quad\mbox{for}\quad i=r+1,\ldots,\min\{n_1,n_2\}.
\end{equation}
Since the zero singular tubes do not affect the t-product, we use the best tubal-rank-$r$ approximation to express STHT as $\cH_r(\cX)=\cX_r$.
\end{definition}

\subsection{Functions on a Tensor Space}
In this section, we define novel concepts for a class of functions on a tensor space. By treating a tensor as a multi-dimensional array, we can define a differentiable function on a tensor space. Specifically, a function $f:\mathbb{R}^{n_1\times n_2\times n_3}\to\mathbb{R}$ is called differentiable if all partial derivatives $\frac{\partial f}{\partial \cX_{i,j,k}}$ exist, and in addition the gradient is simply given by $\nabla f(\cX)=(\frac{\partial f}{\partial \cX_{i,j,k}})_{n_1\times n_2\times n_3}$. Alternatively, we could first convert a tensor function to a multivariate one, compute its gradient and then rewrite it in tensor form. For example, if $f(\cX)=\langle \cA,\cX\rangle$ with $\cA,\cX\in\mathbb{R}^{n_1\times n_2\times n_3}$, then $\nabla f(\cX)=\cA$.

\begin{definition}\label{def:tRSC}
The function $f:\mathbb{R}^{n_1\times n_2\times n_3}\to\mathbb{R}$ satisfies the tubal-rank restricted strong convexity (tRSC) if there exists $\rho_{r}^->0$ such that
\begin{equation}\label{eqn:tRSC}
f(\cX_2)-f(\cX_1)-\Big\langle \nabla f(\cX_1),\cX_2-\cX_1\Big\rangle\geq \frac{\rho_{r}^-}2\norm{\cX_2-\cX_1}^2
\end{equation}
for $\cX_1,\cX_2\in\mathbb{R}^{n_1\times n_2\times n_3}$ with $\rank_t(\cX_1-\cX_2)\leq r$.
\end{definition}

\noindent\textbf{Remark.} Based on the definition of tRSC, we switch the role of $\cX_1$ and $\cX_2$ in \eqref{eqn:tRSC} and get
\begin{equation}\label{eqn:tRSCv2}
\langle \nabla f(\cX_2)-\nabla f(\cX_1),\cX_2-\cX_1\rangle \geq \rho_r^-\norm{\cX_2-\cX_1}^2.
\end{equation}

\begin{definition}\label{def:tRSS}
The function $f:\mathbb{R}^{n_1\times n_2\times n_3}\to \mathbb{R}$ satisfies the tubal-rank restricted strong smoothness (tRSS) if there exists $\rho_{r}^+>0$ such that
\begin{equation}\label{eqn:tRSS}
\norm{\nabla f(\cX_1)-\nabla f(\cX_2)}\leq \rho_{r}^+\norm{\cX_1-\cX_2}
\end{equation}
for $\cX_1,\cX_2\in\mathbb{R}^{n_1\times n_2\times n_3}$ with $\rank_t(\cX_1-\cX_2)\leq r$. Note that $\rho_r^+$ becomes the Lipschitz constant of $\nabla f$ when there is no restriction on the tubal-rank of $\cX_1$ and $\cX_2$.
\end{definition}

\noindent\textbf{Remark.} If $f$ satisfies the tRSS property and $\Omega=\mathrm{span}(\cX_1,\cX_2)$, i.e., the tensor space linearly spanned by $\cX_1$ and $\cX_2$, we can follow the proofs in \cite{nguyen2017linear,qin2021stochastic} to show the tubal-rank restricted co-coercivity
\begin{equation}\label{eqn:cocoer}
\norm{\cP_\Omega(\nabla f(\cX_2)-\nabla f(\cX_1))}^2\leq
\rho_r^+\langle \nabla f(\cX_1)-\nabla f(\cX_2),\cX_1-\cX_2\rangle
\end{equation}
for $\cX_1,\cX_2\in\mathbb{R}^{n_1\times n_2\times n_3}$ with $\rank_t(\cX_1-\cX_2)\leq r$.
Here $\cP_\Omega(\cdot)$ projects a tensor onto the linear space $\Omega$.

\begin{definition}\label{tRIP}\cite{zhang2019rip}
Consider a linear map $\bm\theta:\mathbb{R}^{n_1\times n_2\times n_3}\to\mathbb{R}^m$.
If there exists a constant $\delta_r\in(0,1)$ such that
\begin{equation}\label{eqn:tRIP}
(1-\delta_r)\norm{\cX}^2\leq \norm{\bm\theta(\cX)}^2\leq (1+\delta_r)\norm{\cX}^2
\end{equation}
for any tensor $\cX\in\mathbb{R}^{n_1\times n_2\times n_3}$ whose tubal rank is at most $r$, then $\bm\theta$ is said to satisfy the tensor-tubal-rank $r$ restricted isometry condition (tRIP) with the restricted isometry constant $\delta_r$.
\end{definition}

\begin{lemma}\label{lem:tRIP}
If $f(\cX)=\frac12\norm{\bm\theta(\cX)-\vy}^2$ where $\vy\in\mathbb{R}^m$ and $\bm\theta:\mathbb{R}^{n_1\times n_2\times n_3}\to\mathbb{R}^m$ is a linear map satisfying the tRIP with $\delta_r$, then we have
\[
\rho_r^-=1-\delta_r,\quad \rho_r^+=1+\delta_r.
\]
\end{lemma}
The proof of this lemma is based on the matrix representation of the linear map $\bm\theta$, i.e., a matrix $A\in\mathbb{R}^{m\times (n_1n_2n_3)}$ exists with $A\,\vect{\cX}=\vect{\bm\theta(\cX)}$ where $\vect{\cdot}$ is an operator that converts a tensor to a vector by column-wise stacking at each frontal slice followed by the slice stacking.

\section{Proposed Algorithms}\label{sec:alg}
\subsection{Iterative Singular Tube Hard Thresholding}
Let $f:\mathbb{R}^{n_1\times n_2\times n_3}\to\mathbb{R}$ be a differentiable function. Consider the minimization problem
\begin{equation}
\min_{\cX\in\mathbb{R}^{n_1\times n_2\times n_3}} f(\cX)\quad\st\quad \rank_t(\cX)\leq r,
\end{equation}
which can also be written as
\[
\min_{\cX\in\mathbb{T}_r} f(\cX)
\]
where $\mathbb{T}_r$ is defined in \eqref{eqn:Tr}. Following the idea of the alternating minimization algorithm, we can obtain an algorithm that alternates unconstrained minimization of $f$ and projection to $\mathbb{T}_r$. Thus,
based on the STHT operator defined in Definition~(\ref{def:STHT}), we propose the iterative singular tube hard thresholding algorithm (ISTHT) in Algorithm~\ref{alg:TIHT}, which alternates gradient descent and singular tube hard thresholding using t-SVD. Empirically, we have observed that t-SVD returns an f-diagonal tensor whose first frontal slice has much larger diagonal entries than those in the remaining frontal slices, thus we recommend scaling the input tensor before performing STHT and then re-scaling back. Furthermore, the step size $\gamma_t$ for gradient descent can be set as a constant or defined by an adaptive method during the iterations, such as line search, a trust region method or the Barzilai-Borwein method \cite{barzilai1988two}. Moreover, the Nesterov acceleration technique \cite{nesterov1983method} has been integrated into gradient descent, i.e., Nesterov Accelerated Gradient Descent, with further developments of popular momentum-based methods in deep learning, such as AdaGrad \cite{duchi2011adaptive}. It has been shown that Nesterov Accelerated Gradient Descent can achieve a convergence rate with $O(1/N^2)$ while gradient descent typically converges with the rate $O(1/N)$ where $N$ is the number of iterations \cite{nesterov1983method}. More specifically, at the $t$-th iteration, the step size $\gamma_t$ is updated by
\[
\lambda_{t+1}=\frac{1+\sqrt{1+4\lambda_{t}^2}}{2},\quad \gamma_{t+1}=\frac{1-\lambda_{t}}{\lambda_{t+1}},\quad\mbox{for } t=0,1,\ldots,N-1,
\]
with the initial $\lambda_0=1$. The gradient descent step thereby becomes the linear combination of the original gradient descent step and the previous step with weight specified by {$\gamma_{t+1}$}. The detailed algorithm is summarized in Algorithm~\ref{alg:aTIHT}. All the algorithms terminate when either the maximum number of iterations is reached or the stopping criterion $\norm{\cX^{t+1}-\cX^t}/\norm{\cX^t}<tol$ with a preassigned tolerance $tol$ is met.
Empirical results in Section~\ref{sec:exp} will show that properly selected step size regime may significantly improve accuracy and convergence.

\begin{algorithm}
\caption{Iterative Singular Tube Hard Thresholding (ISTHT)}\label{alg:TIHT}
\begin{algorithmic}
\State\textbf{Input:} tubal rank $r$, step size $\gamma$, tolerance $tol$.
\State\textbf{Output:} $\cX^{t+1}$.
\State\textbf{Initialize:} $\cX^0=\mathbf{0}\in\mathbb{R}^{n_1\times n_2\times n_3}$.
\For {$t=0,1,\ldots,N-1$}
\State $\cZ^{t+1}=\cX^t-\gamma\nabla f(\cX^t)$
\State $\cX^{t+1}=\mathcal{H}_r(\cZ^{t+1})$
\State Exit if the stopping criterion is met.
\EndFor
\end{algorithmic}
\end{algorithm}

\begin{algorithm}
\caption{Accelerated Iterative Singular Tube Hard Thresholding (aISTHT)}\label{alg:aTIHT}
\begin{algorithmic}
\State\textbf{Input:} tubal rank $r$, {step size $\gamma$,} tolerance $tol$.
\State\textbf{Output:} $\cX^{t+1}$.
\State\textbf{Initialize:} $\cX^0=\widetilde{\cZ}^0=\mathbf{0}\in\mathbb{R}^{n_1\times n_2\times n_3}$, $\lambda_0=1$.
\For {$t=0,1,\ldots,N-1$}
\State $\lambda_{t+1}=\frac{1+\sqrt{1+4\lambda_{t}^2}}{2}$
\State $\gamma_{t+1}=\frac{1-\lambda_{t}}{\lambda_{t+1}}$
\State $\widetilde{\cZ}^{t+1}=\cX^t-{\gamma}\nabla f(\cX^t)$
\State $\cZ^{t+1}=(1-\gamma_{t+1})\widetilde{\cZ}^{t+1}+\gamma_{t+1}\widetilde{\cZ}^t$
\State $\cX^{t+1}=\mathcal{H}_r(\cZ^{t+1})$
\State Exit if the stopping criterion is met.
\EndFor
\end{algorithmic}
\end{algorithm}

\subsection{Stochastic Iterative Singular Tube Hard Thresholding}
Consider a collection of functions $f_j:\mathbb{R}^{n_1\times n_2\times n_3}\to\mathbb{R}$ with $j=1,\ldots,M$ and their average
\[
F(\cX)=\frac1M\sum_{j=1}^Mf_j(\cX),\quad \cX\in\mathbb{R}^{n_1\times n_2\times n_3}.
\]
Now we consider the following tubal-rank constrained minimization problem
\begin{equation}\label{eqn:model}
\min_{\cX\in\mathbb{R}^{n_1\times n_2\times n_3}}F(\cX)\quad\st\quad \rank_t(\cX)\leq r.
\end{equation}
By combining the stochastic gradient descent and best tubal-rank approximation steps, we proposed the Stochastic Iterative Singular Tube Hard Thresholding (StoISTHT) in Algorithm~\ref{alg:StoTIHT}.

\begin{algorithm}
\caption{Stochastic Iterative Singular Tube Hard Thresholding (StoISTHT)}\label{alg:StoTIHT}
\begin{algorithmic}
\State\textbf{Input:} tubal rank $r$, step size $\gamma$, tolerance $tol$, probabilities $\{p(i)\}_{i=1}^M$.
\State\textbf{Output:} $\cX^{t+1}$.
\State\textbf{Initialize:} $\cX^0=\mathbf{0}\in\mathbb{R}^{n_1\times n_2\times n_3}$.
\For {$t=0,1,\ldots,N-1$}
\State Randomly select an index $i_t\in\{1,2,\ldots,M\}$ with probability $p(i_t)$
\State $\cZ^{t+1}=\cX^t-\frac{\gamma}{Mp(i_t)}\nabla f_{i_t}(\cX^t)$
\State $\cX^{t+1}=\mathcal{H}_r(\cZ^{t+1})$
\State Exit if the stopping criterion is met.
\EndFor
\end{algorithmic}
\end{algorithm}

Based on the mini-batch technique \cite{needell2016batched}, we propose an accelerated version--Batched Stochastic Iterative Singular Tube Hard Thresholding (BStoISTHT), summarized in Algorithm~\ref{alg:BStoTIHT}.
\begin{algorithm}
\caption{Batched Stochastic Iterative Singular Tube Hard Thresholding (BStoISTHT)}\label{alg:BStoTIHT}
\begin{algorithmic}
\State\textbf{Input:} tubal rank $r$, step size $\gamma$, tolerance $tol$, {batch probabilities $\{p(\tau)\}$}.
\State\textbf{Output:} $\cX^{t+1}$.
\State\textbf{Initialize:} $\cX^0=\mathbf{0}\in\mathbb{R}^{n_1\times n_2\times n_3}$.
\For {$t=0,1,\ldots,N-1$}
\State Randomly select an index batch $\tau_t\subseteq\{1,2,\ldots,d\}$ of size $b$ with probability $p(\tau_t)$
\State $\cZ^{t+1}=\cX^t-\frac{\gamma}{Mp(\tau_t)}(\frac1b\sum_{j\in\tau_t}\nabla f_{j}(\cX^t))$
\State $\cX^{t+1}=\mathcal{H}_r(\cZ^{t+1})$
\State Exit if the stopping criterion is met.
\EndFor
\end{algorithmic}
\end{algorithm}

\section{Convergence Analysis}\label{sec:cnvg}
In this section, we provide the convergence analysis for the proposed algorithms, which can are extended from the matrix case \cite{nguyen2017linear,qin2021stochastic} to the more general tensor setting \cite{rauhut2017low}.
\begin{lemma}\label{lem1}
Let $\cX,\cY\in\mathbb{R}^{n_1\times n_2\times n_3}$ with tubal-rank $r$. Then we have
\[
\norm{\cH_r(\cX)-\cY}^2\leq 2\langle \cH_r(\cX)-\cY,\cX-\cY\rangle.
\]
\end{lemma}

\begin{proof}
Since the operator $\cH_r$ gives the best tubal-rank-$r$ approximation, we have
\[
\norm{\cH_r(\cX)-\cX}\leq \norm{\cY-\cX}.
\]
Then we get
\[
\begin{aligned}
\norm{\cH_r(\cX)-\cY}^2&=\norm{\cH_r(\cX)-\cX+\cX-\cY}^2\\
&=\norm{\cH_r(\cX)-\cX}^2+\norm{\cX-\cY}^2+2\langle \cH_r(\cX)-\cX,\cX-\cY\rangle\\
&\leq 2\norm{\cY-\cX}^2+2\langle \cH_r(\cX)-\cX,\cX-\cY\rangle\\
&=2\langle \cH_r(\cX)-\cY,\cX-\cY\rangle.
\end{aligned}
\]

\end{proof}

\begin{theorem}\label{thm:alg1conv}
Let $f:\mathbb{R}^{n_1\times n_2\times n_3}\to\mathbb{R}^m$ satisfy the tRSC and tRSS, and $\cX^*$ be a minimizer of $f$ with the tubal-rank at most $r$. Then there exist $\kappa,\sigma>0$ such that the recovery error at the $t$-th iteration of Algorithm~\ref{alg:TIHT} is bounded from above
\[
\norm{\cX^t-\cX^*}\leq \kappa^{{t}}\norm{\cX^0-\cX^*}+\frac{\sigma}{1-\kappa}.
\]
\end{theorem}

\begin{proof}

Let $\cR^t=\cX^t-\cX^*$, $\Omega$ be the linear space spanned by $\cX^*$, $\cX^t$ and $\cX^{t+1}$, and $\cP_\Omega$ be the orthogonal projection from $\mathbb{R}^{n_1\times n_2\times n_3}$ to $\Omega$. {Note that $\Omega$ and $\cP_\Omega$ may change over the iterations.} Then for any $\cX\in\mathbb{R}^{n_1\times n_2\times n_3}$, $\cP_\Omega(\cX)$ must have tubal rank $\leq 3r$.

By Lemma~\ref{lem1}, we calculate the norm square of ${\cR}^{t+1}$ as follows
\[
\begin{aligned}
\norm{\cR^{t+1}}^2&\leq2\langle \cX^{t+1}-\cX^*,\cZ^{t+1}-\cX^*\rangle\\
&=2\langle \cX^{t+1}-\cX^*,\cX^{t}-\gamma\nabla f(\cX^t)-\cX^*\rangle\\
&=2\langle \cX^{t+1}-\cX^*,\cX^{t}-\cX^*-\gamma(\nabla f(\cX^t)-\nabla f(\cX^*))\rangle
-2\langle \cX^{t+1}-\cX^*,\gamma\nabla f(\cX^*)\rangle
\end{aligned}
\]
The definition of the orthogonal projection $\cP_\Omega$ yields that
\[
\langle \cR^{t+1},\cP_{\Omega^c}(\nabla f(\cX^t)-\nabla f(\cX^*))\rangle =0,\quad
\langle \cR^{t+1},\cP_{\Omega^c}\nabla f(\cX^*)\rangle =0.
\]
Thus we have
\[
\begin{aligned}
\norm{\cR^{t+1}}^2&\leq 2\langle \cR^{t+1},\cR^t-\gamma\cP_\Omega\big(\nabla f(\cX^t)-\nabla f(\cX^*)\big)\rangle
-2\langle \cR^{t+1},\gamma\cP_\Omega\nabla f(\cX^*)\rangle\\
&\leq 2\norm{\cR^{t+1}}\left(\norm{\cR^t-\gamma\cP_\Omega\big(\nabla f(\cX^t)-\nabla f(\cX^*)\big)}
+\norm{\gamma\cP_\Omega\nabla f(\cX^*)}\right)
\end{aligned}
\]
which implies that
\[
\norm{\cR^{t+1}}\leq 2\norm{\cR^t-\gamma\cP_\Omega\big(\nabla f(\cX^t)-\nabla f(\cX^*)\big)}
+2\gamma\norm{\cP_\Omega\nabla f(\cX^*)}.
\]
By using the co-coercivity~\eqref{eqn:cocoer} and the reformulation \eqref{eqn:tRSCv2} of tRSC, we have
\[\begin{aligned}
&\norm{\cR^t-\gamma\cP_\Omega\big(\nabla f(\cX^t)-\nabla f(\cX^*)\big)}^2\\
&=\norm{\cR^t}^2+\gamma^2\norm{\cP_\Omega\big(\nabla f(\cX^t)-\nabla f(\cX^*)\big)}^2-2\gamma\langle \cR^t, \cP_\Omega\big(\nabla f(\cX^t)-\nabla f(\cX^*)\big)\rangle\\
&\leq \norm{\cR^t}^2+\gamma^2\rho_s^+\langle \cR^t,\nabla f(\cX^t)-\nabla f(\cX^*)\rangle -2\gamma\langle \cR^t,\nabla f(\cX^t)-\nabla f(\cX^*)\rangle\\
&\leq \norm{\cR^t}^2-(2\gamma-\gamma^2\rho_{3r}^+)\langle \cR^t,\nabla f(\cX^t)-\nabla f(\cX^*)\rangle\\
&\leq \left(1-(2-\gamma\rho_{3r}^+)\gamma\rho_{3r}^-\right)\norm{\cR^t}^2.
\end{aligned}\]

Therefore, we get
\[
\norm{\cR^{t+1}}\leq 2\sqrt{1-(2-\gamma\rho_{3r}^+)\gamma\rho_{3r}^-}\norm{\cR^t}+2\gamma\norm{\cP_\Omega\nabla f(\cX^*)}:=\kappa\norm{\cR^t}+\sigma.
\]
By recursively applying the above inequality, we get
\[
\norm{\cR^{t}}\leq \kappa^t\norm{\cR^0}+\sigma\sum_{i=0}^{t-2}\kappa^i\leq \kappa^t\norm{\cR^0}+\frac{\sigma}{1-\kappa}
\]
which completes the proof.
\end{proof}

\begin{theorem}\label{thm:alg1linear}
Let $f(\cX)=\frac1{2}\norm{\cA(\cX)-\vy}^2$ with $\vy\in\mathbb{R}^m$. If the linear map $\cA:\mathbb{R}^{n_1\times n_2\times n_3}\to\mathbb{R}^m$ satisfies the tRIP with the restricted isometry constant $\delta_s$ and the measurements $\vy=\cA(\cX^*)+\ve$ with $\rank_t(\cX^*)\leq r$ and $\norm{\ve}\leq\varepsilon$, then the $t$-th iteration of Algorithm~\ref{alg:TIHT} has a bounded recovery error
\[
\norm{\cX^t-\cX^*}\leq \kappa^t\norm{\cX^0-\cX^*}+\frac{\sigma}{1-\kappa},
\]
where
\begin{equation}\label{eqn:alg1coef}
\kappa = 2(\mid 1-\gamma\mid +\gamma\delta_{3r}),\quad\sigma = 2\gamma\varepsilon\sqrt{1+\delta_{2r}}.
\end{equation}
\end{theorem}

\begin{proof}
First the gradient of $f$ is given by
\[
\nabla f(\cX^t)=\cA^*(\cA(\cX^t)-\vy)
\]
where the adjoint operator $\cA^*:\mathbb{R}^m\to\mathbb{R}^{n_1\times n_2\times n_3}$ is defined by $\langle \cA(\cX),\vz\rangle =\langle \cX,\cA^*(\vz)\rangle$ for any $\cX\in\mathbb{R}^{n_1\times n_2\times n_3}$ and $\vz\in\mathbb{R}^m$.

Let $\cR^{t}=\cX^t-\cX^*$, $\Omega=\mathrm{span}(\cX^*,\cX^t,\cX^{t+1})$, and $\cP_{\Omega}$ be the orthogonal projection from $\mathbb{R}^{n_1\times n_2\times n_3}$ onto $\Omega$. To simplify the notation, we denote
\[
\cA_{\Omega}(\cX)=\cA(\cP_\Omega(\cX))=\cA\circ\cP_\Omega(\cX),\quad \cX\in\mathbb{R}^{n_1\times n_2\times n_3}.
\]
Then we have $\cA_{\Omega} (\cR^{t+1})=\cA(\cR^{t+1})$ and $\cA_{\Omega}(\cR^t)=\cA(\cR^t)$. Note that the operator $\cA_{\Omega}$ is different from the composite operator $\cP_\Omega\circ \nabla f$ even when $\nabla f$ is a linear map. By Lemma~\ref{lem1}, we estimate $\norm{\cR^{t+1}}$ as follows
\[
\begin{aligned}
\norm{\cR^{t+1}}^2&\leq 2\langle \cR^{t+1},\cZ^{t+1}-\cX^*\rangle\\
&=2\langle \cR^{t+1},\cX^t-\gamma\cA^*(\cA(\cX^t)-\vy)-\cX^*\rangle\\
&=2\langle \cR^{t+1},\cX^t-\cX^*-\gamma\cA^*(\cA(\cX^t)-\cA(\cX^*))\rangle
+2\langle \cR^{t+1},\gamma\cA^*(\ve)\rangle\\
&=2\langle \cR^{t+1},(\cI-\gamma\cA^*\cA)(\cR^t)\rangle +2\gamma\langle \cA(\cR^{t+1}),\ve\rangle\\
&=2\langle \cR^{t+1},(\cI-\gamma\cA_{\Omega}^*\cA_{\Omega})(\cR^t)\rangle +2\gamma\langle \cA(\cR^{t+1}),\ve\rangle\\
&\leq 2\norm{\cR^{t+1}}\left(\norm{(\cI-\gamma\cA_{\Omega}^*\cA_{\Omega})(\cR^t)}
+\gamma\varepsilon\sqrt{1+\delta_{2r}}\right){,}
\end{aligned}
\]
where $\cI:{\mathbb{R}^{n_1\times n_2\times n_3}}\to{\mathbb{R}^{n_1\times n_2\times n_3}}$ is the identity map. The last inequality follows from applying the Cauchy-Schwarz inequality to each term in the summation, the tRIP assumption on $\mathcal{A}$ for the second term (note that $\mathcal{R}^{t+1}$ has tubal-rank at most $2r$), and lastly, the assumption that $\|\ve\|\leq \epsilon$.

Therefore we have
\[\begin{aligned}
\norm{\cR^{t+1}}&\leq 2\norm{(\cI-\gamma\cA_{\Omega}^*\cA_{\Omega})(\cR^t)}+2\gamma\varepsilon\sqrt{1+\delta_{2r}}\\
&\leq 2\norm{(1-\gamma)\cR^t}+2\gamma\norm{(\cI-\cA_{\Omega}^*\cA_{\Omega})(\cR^t)}+2\gamma\varepsilon\sqrt{1+\delta_{2r}}\\
&\leq 2\left(\mid 1-\gamma\mid +\gamma\delta_{3r}\right)\norm{\cR^t}+2\gamma\varepsilon\sqrt{1+\delta_{2r}}.
\end{aligned}\]
The last inequality holds due to the fact that the operator norm of the map $\cI-\cA_{\Omega}^*\cA_\Omega$ can be bounded from above by
\[
\norm{\cI-\cA_{\Omega}^*\cA_{\Omega}}\leq \sup_{\cX:\rank_t(\cX)\leq 3r\atop \norm{\cX}=1}\mid\norm{\cX}^2-\norm{\cA(\cX)}^2\mid=\delta_{3r}.
\]
Thus we get
\[
\norm{\cR^{t+1}}\leq \kappa \norm{\cR^t}+\sigma,
\]
where $\kappa = 2(\mid 1-\gamma\mid+\gamma\delta_{3r})$ and $\sigma=2\gamma\varepsilon\sqrt{1+\delta_{2r}}$. By the recursive relationship, we get
\[
\norm{\cR^t}\leq \kappa^t\norm{\cR^0}+\frac{\sigma}{1-\kappa},
\]
provided that $\kappa<1$, which completes the proof.

\end{proof}

Based on the tRIP of sub-Gaussian ensembles in \cite{zhang2019tensor} and Theorem~\ref{thm:alg1linear}, we get the following corollary about the convergence of Algorithm~\ref{alg:TIHT} for the sub-Gaussian measurements.
\begin{corollary}\label{cor:subgauss}
Let $f(\cX)=\cA(\cX)$ where $\cA:\mathbb{R}^{n_1\times n_2\times n_3}\to\mathbb{R}^m$ is a linear sub-Gaussian measurement ensemble, and the solution $\cX^*$ with tubal-rank $r$ satisfy $\vy=\cA(\cX^*)+\ve$. If $\gamma=1$ and the number of measurements satisfies
\[
m\geq C\max\{r(n_1+n_2+1)n_3,\log(\epsilon^{-1})\},
\]
with $C>0$ which depends on the sub-Gaussian parameter, then Algorithm~\ref{alg:TIHT} converges to the solution $\cX^*$ with probability at least $1-\epsilon$.

\end{corollary}

\begin{theorem}\label{thm:alg2conv}
Let $\cX^*$ be a feasible solution of \eqref{eqn:model} and $\cX^0$ be the initial guess. Assume that $F$ satisfies tRSC with the constant $\rho_r^{-}$ and each $f_i$ satisfies tRSS with the constant $\rho_r^+(i)$. Then there exist a contraction coefficient $\kappa>0$ and a tolerance coefficient $\sigma>0$ such that the expectation of the recovery error at the $t$-th iteration of Algorithm~\ref{alg:StoTIHT} is bounded from above via
\begin{equation}\label{eqn:SMVStoIHT_E}
E\norm{\cX^{t}-\cX^*}\leq \kappa^{t}\norm{\cX^0-\cX^*}+\sigma.
\end{equation}
If $\kappa<1$, then Algorithm~\ref{alg:StoTIHT} generates a sequence $\{\cX^t\}$ which converges to the desired solution $\cX^*$.
\end{theorem}

\begin{proof}
Let $\cR^t=\cX^t-\cX^*$, $\Omega$ be the linear space spanned by $\cX^*$, $\cX^t$ and $\cX^{t+1}$, and $\cP_\Omega$ be the orthogonal projection from $\mathbb{R}^{n_1\times n_2\times n_3}$ to $\Omega$. Then for any $\cX\in\mathbb{R}^{n_1\times n_2\times n_3}$, $\cP_\Omega(\cX)$ must have tubal rank $\leq 3r$.

By Lemma~\ref{lem1}, we calculate the norm square of $R^{t+1}$ as follows
\[
\begin{aligned}
\norm{\cR^{t+1}}^2&\leq2\langle \cX^{t+1}-\cX^*,\cZ^{t+1}-\cX^*\rangle\\
&=2\langle \cX^{t+1}-\cX^*,\cX^{t}-\frac{\gamma}{Mp(i_t)}\nabla f_{i_t}(\cX^t)-\cX^*\rangle\\
&=2\langle \cX^{t+1}-\cX^*,\cX^{t}-\cX^*-\frac{\gamma}{Mp(i_t)}(\nabla f_{i_t}(\cX^t)-\nabla f_{i_t}(\cX^*))\rangle\\
&\phantom{=}-2\langle \cX^{t+1}-\cX^*,\frac{\gamma}{Mp(i_t)}\nabla f_{i_t}(\cX^*)\rangle
\end{aligned}
\]
The definition of the orthogonal projection $\cP_\Omega$ yields that
\[
\langle \cR^{t+1},\cP_{\Omega^c}(\nabla f_{i_t}(\cX^t)-\nabla f_{i_t}(\cX^*))\rangle =0,\quad
\langle \cR^{t+1},\cP_{\Omega^c}\left(\nabla f_{i_t}(\cX^*)\right)\rangle =0.
\]
Thus we have
\[
\begin{aligned}
&\phantom{\leq} \norm{\cR^{t+1}}^2\\
&\leq 2\langle \cR^{t+1},\cR^t-\frac{\gamma}{Mp(i_t)}\cP_\Omega\big(\nabla f_{i_t}(\cX^t)-\nabla f_{i_t}(\cX^*)\big)\rangle
-2\langle \cR^{t+1},\frac{\gamma}{Mp(i_t)}\cP_\Omega\nabla f_{i_t}(\cX^*)\rangle\\
&\leq 2\norm{\cR^{t+1}}\left(\norm{\cR^t-\frac{\gamma}{Mp(i_t)}\cP_\Omega\big(\nabla f_{i_t}(\cX^t)-\nabla f_{i_t}(\cX^*)\big)}+\norm{\frac{\gamma}{Mp(i_t)}\cP_\Omega\nabla f_{i_t}(\cX^*)}\right).
\end{aligned}
\]
Let $I_t=\{i_1,\ldots,i_t\}$ be the set of indices that are randomly drawn from the discrete distribution $\{p(i)\}_{i=1}^M$ after $t$ iterations. By taking the expectation on both sides of the above inequality with respect to $i_t$ conditioned on $I_{t-1}$, we get
\[\begin{aligned}
E_{i_t\mid I_{t-1}}\norm{\cR^{t+1}}&\leq 2E_{i_t\mid I_{t-1}}\norm{\cR^t-\frac{\gamma}{Mp(i_t)}\cP_\Omega\big(\nabla f_{i_t}(\cX^t)-\nabla f_{i_t}(\cX^*)\big)}\\
&\phantom{\leq} +2E_{i_t\mid I_{t-1}}\norm{\frac{\gamma}{Mp(i_t)}\cP_\Omega\nabla f_{i_t}(\cX^*)}.
\end{aligned}\]
By adapting the proof in \cite{nguyen2017linear}, we are able to show that
\[
E_{i_t\mid I_{t-1}}\norm{\cR^t-\frac{\gamma}{Mp(i_t)}\cP_\Omega\big(\nabla f_{i_t}(\cX^t)-\nabla f_{i_t}(\cX^*)\big)}
\leq \sqrt{1-(2-\gamma\alpha_{3r})\gamma\rho_{3r}^-}\norm{\cR^t},
\]
where $\alpha_{3r}=\max_i\frac{\rho_{3r}^+(i)}{Mp(i)}$. In addition, we have
\[
E_{i_t\mid I_{t-1}}\norm{\frac{\gamma}{Mp(i_t)}\cP_\Omega\nabla f_{i_t}(\cX^*)}\leq \frac{\gamma}{M\max_ip(i)}E_{i_t}\norm{\cP_\Omega f_{i_t}(\cX^*)}.
\]
Then we get
\[
E_{i_t\mid I_{t-1}}\norm{\cR^{t+1}}\leq \kappa \norm{\cR^t}+\sigma,
\]
where
\[
 \kappa=2\sqrt{1-(2-\gamma\alpha_{3r})\gamma\rho_{3r}^-},\quad\mbox{and}\quad
 \sigma = \frac{2\gamma}{M\max_ip(i)}E_{i_t}\norm{\cP_\Omega f_{i_t}(\cX^*)}.
\]
By taking the expectation on both sides with respect to $I_{t-1}$, we get the desired inequality
\[
E_{I_t}\norm{\cR^{t+1}}\leq \kappa E_{I_{t-1}}\norm{\cR^t}+\sigma,
\]
which further yields
\[
E_{I_t}\norm{\cX^t-\cX^*}\leq \kappa^t\norm{\cX^0-\cX^*}+\frac{\sigma}{1-\kappa}.
\]
Applying this result recursively completes the proof.
\end{proof}

\section{Numerical Experiments}\label{sec:exp}

In this section, we show how to apply the proposed algorithms to solve multiple application problems, including tensor compressive sensing recovery and low-tubal-rank tensor completion, and conduct a variety of numerical experiments to show the performance of the proposed algorithms. For the performance comparison metric, we use the recovery error (RE) for tensor recovery defined by
\[
RE=\|\cX-\widetilde{\cX}\|/\norm{\cX}
\]
where $\widetilde{\cX}$ is an estimation of the ground truth $\cX$. All numerical experiments were run in MATLAB R2021b on a desktop computer with 64GB RAM and a 3.10GHz Intel Core i9-9960X CPU.

\subsection{Tensor Compressive Sensing}
By extending the matrix case, we consider the linear tensor compressive sensing problem where each measurement is generated by the inner product of two tensors, i.e., a sensing tensor and a low-tubal-rank tensor to be reconstructed. Let $\cX\in\mathbb{R}^{n_1\times n_2\times n_3}$ be the tensor with low tubal rank. Given a collection of sensing tensors $\{\cA_j\}_{j=1}^{M}$, the obtained data vector $\vy\in\mathbb{R}^M$ is generated by a linear map $\bm\theta$ from $\mathbb{R}^{n_1\times n_2\times n_3}$ to $\mathbb{R}^M$ given by
\begin{equation}\label{eqn:A}
\bm\theta(\cX)=\begin{bmatrix}\langle \cA_1,\cX\rangle \\ \vdots\\ \langle \cA_M,\cX\rangle\end{bmatrix},\quad \cA_j\in\mathbb{R}^{n_1\times n_2\times n_3},\quad j=1,\ldots,M.
\end{equation}
Let $\vect{\cdot}$
be the operator that converts a tensor into a vector by columnwise stacking of each frontal slice followed by slicewise stacking. The linear map $\bm\theta$ can be therefore represented as
\[
\bm\theta(\cX)=\begin{bmatrix}\vect{\cA_1}^T \\ \vdots\\ \vect{\cA_{M}}^T\end{bmatrix} \vect{\cX}:=A\,\vect{\cX}
\]
where $A\in\mathbb{R}^{M\times (n_1n_2n_3)}$ and $\vect{\cX}\in\mathbb{R}^{n_1n_2n_3}$. Then the compressive sensing low-tubal-rank tensor recovery problem boils down to solving \eqref{eqn:model} with the following objective function
\begin{equation}\label{eqn:CSobj}
F(\cX)=\frac1{2M}\norm{\bm\theta(\cX)-\vy}^2=\frac1M\sum_{\ell=1}^M{\frac12\Big| \sum_{i,j,k} \cA^\ell_{ijk}\cX_{ijk}-y_\ell\Big|^2}:=\frac1M\sum_{\ell=1}^Mf_\ell(\cX).
\end{equation}
As shown in \cite{zhang2019rip}, sub-Gaussian measurements satisfy tRIP and thereby satisfy tRSC and tRSS. Therefore, our proposed algorithms can be applied with guaranteed convergence.

To justify the performance of the proposed algorithms, we test synthetic data, where the sensing matrix $A$ and the ground truth $\cX$ with low tubal rank consist of independently identically distributed (i.i.d.) samples from the Gaussian distribution.
In our first set of experiments, we compare the Algorithms~\ref{alg:TIHT} and \ref{alg:aTIHT} with various tensor tubal ranks, measurement sampling rates, and noise ratios. Specifically, all ground truth tensors are of the size $20\times 20\times 10$, the maximum number of iterations is set as 500 and the stepsize $\tau=100$. In Figure~\ref{fig:exp1_rank}, we plot the reconstruction error for the two algorithms versus the iteration number with the tubal rank of $\cX$ ranging in $\{1,2,3,4\}$. By varying the sampling rate {$M/N\in\{0.3,0.4,0.5,0.6\}$ with $N=n_1n_2n_3$} and fixing the tubal rank as 2, we get the results shown in Figure~\ref{fig:exp1_samp}. From these two figures, one can see that Algorithm~\ref{alg:aTIHT} converges faster than Algorithm~\ref{alg:TIHT} in terms of reconstruction error. To further test the robustness to noise, we add to the measurements $\vy$ various types of noise with standard deviation $\max_{1\leq i\leq M}|y_i|{\sigma}$ where ${\sigma}\in\{0.01,0.02,0.03,0.04\}$. The corresponding results are displayed in Figure~\ref{fig:exp1_noise}, which shows that Algorithm~\ref{alg:aTIHT} yields faster decay in the error than Algorithm~\ref{alg:TIHT} but eventually converges to the same limit. Overall, if the tensor rank is relatively low and the sampling rate is high, then both algorithms have fast error decay and Algorithm~\ref{alg:aTIHT} with Nesterov's adaptive stepsize selection strategy effectively accelerates the convergence.
\begin{figure}[h]
\centering
\includegraphics[width=0.62\textwidth]{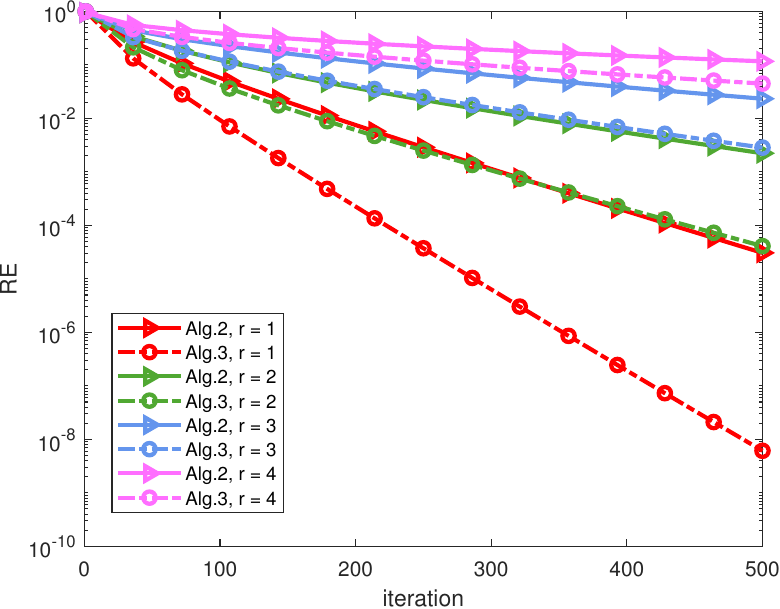}
\caption{Reconstruction error comparison with various tensor tubal ranks. The ground truth is noise-free and the sampling rate is fixed as ${M/N}=0.60$. }\label{fig:exp1_rank}
\end{figure}

\begin{figure}[h]
\centering
\includegraphics[width=0.62\textwidth]{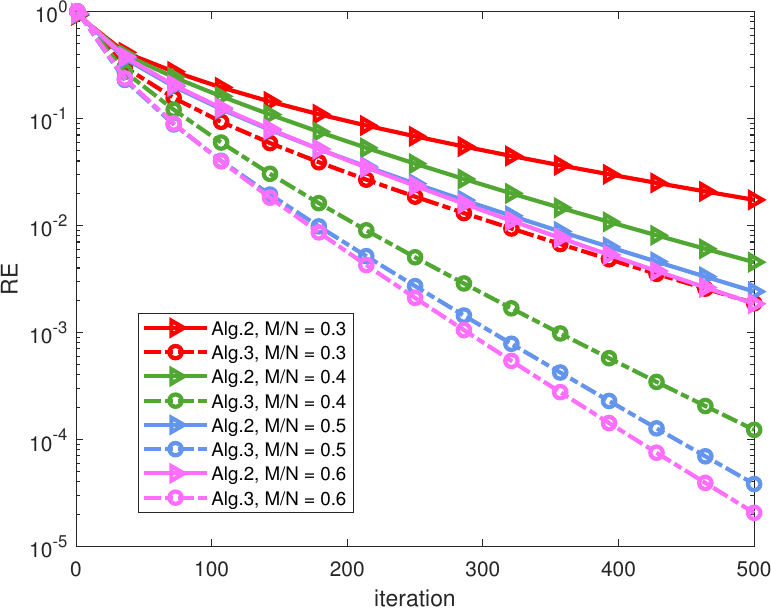}
\caption{Reconstruction error comparison with various sampling rates. The ground truth is noise-free with fixed tubal rank as 2.}\label{fig:exp1_samp}
\end{figure}

\begin{figure}[h]
\centering
\includegraphics[width=0.62\textwidth]{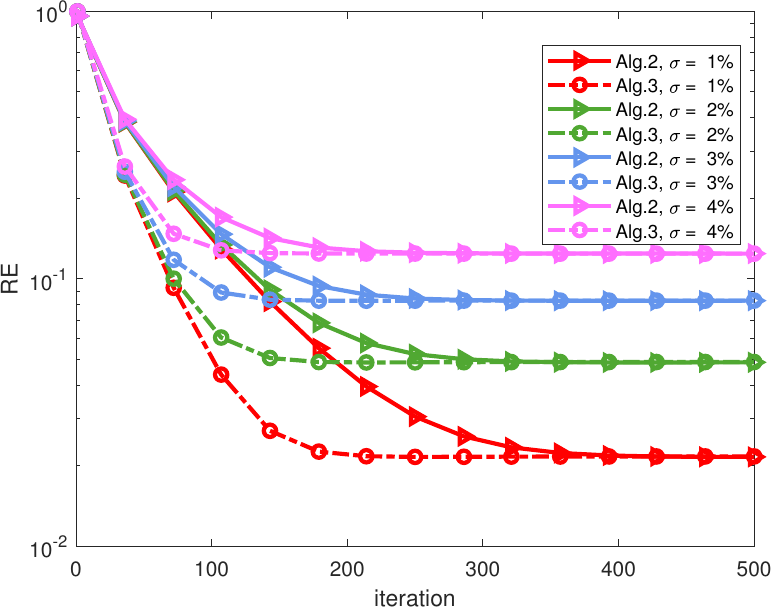}
\caption{Reconstruction error comparison with various noisy measurements.}\label{fig:exp1_noise}
\end{figure}

Furthermore, we test the noise-free tensor data of size $20\times 20\times 10$ with tubal rank one and a fixed sampling rate of 60\%, and apply the Algorithm~\ref{alg:BStoTIHT} with various batch sizes to recover the tensor. The results for using the batch size $b$ ranging in $\{200,400,600,800,1000\}$, i.e., 0.05\% to 0.25\% of the entire tensor size, are shown in Figure~\ref{fig:exp1_alg5_bs}. One can see that the stochastic version of the algorithm takes less iterations but more running time by increasing the batch size. If the batch size increases by 200, then about 4 more seconds in running time will be desired.

\begin{figure}[h]
\centering
\includegraphics[width=0.62\textwidth]{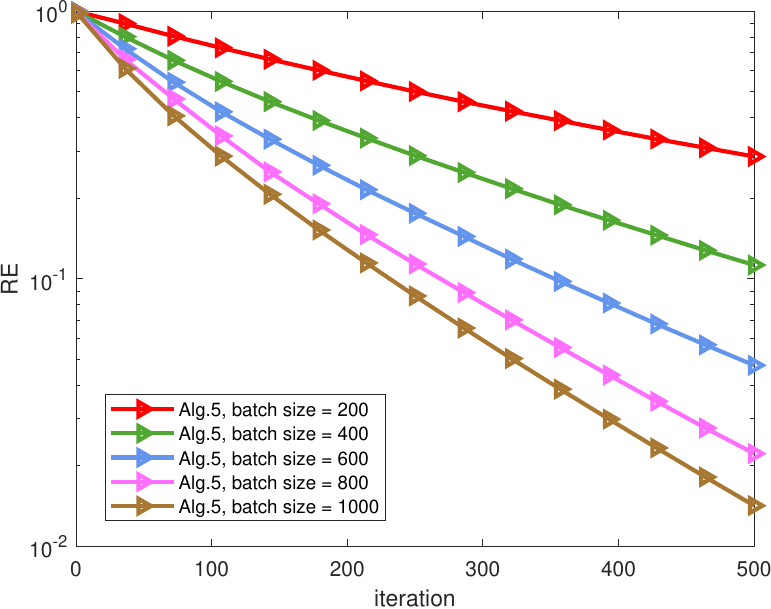}
\caption{Reconstruction error comparison for Algorithm \ref{alg:BStoTIHT} with various batch sizes. The running times for the batch sizes $b=200,400,600,800,1000$ are (in seconds): 8.38, 12.27, 16.51, 20.52, 24.66, respectively.}
\label{fig:exp1_alg5_bs}
\vspace{-10pt}
\end{figure}

\subsection{Color Image Inpainting}
The second application for tensor recovery we show here is color image inpainting. Notice that color images naturally have three-dimensional tensor structures with the third dimension specifying the number of color channels. Assume that a RGB color image $\cX\in\mathbb{R}^{n_1\times n_2\times 3}$ with $n_1\times n_2$ pixels is contaminated by  Gaussian noise with missing pixel intensities. Since the tubal rank of $\cX$ is no greater than each dimension, we consider the following color image inpainting model
\begin{equation}\label{eqn:CIprob}
\min_{\cY\in\mathbb{R}^{n_1\times n_2\times n_3}} \frac12\norm{\cP_{{\Lambda}}(\cX-\cY)}^2,\quad \rank_t(\cY)\leq k.
\end{equation}
Although it can be reformulated as a linear map from $\mathbb{R}^{n_1\times n_2\times n_3}$ to $\mathbb{R}^M$, we express $\cP$ implicitly, which is implemented by extracting entries in {$\Lambda$}. Here we compare the proposed Algorithm ~\ref{alg:TIHT} with two popular color image inpainting methods: (1) Coherence Transport based Inpainting  (CTI) \cite{bornemann2007fast} with Matlab command \verb"inpaintCoherent"; (2) EXemplar-based inpainting in Tensor filling order (EXT) \cite{criminisi2004region,le2013hierarchical} with the Matlab command \verb"inpaintExemplar".

To start with, we create a synthetic color image, i.e., checkerboard image, with size $128\times 128\times 3$ and tubal rank 2. Note that the three color channels have different intensities. The observed image is then obtained by occluding all the pixels from a square of size $80\times 80$ in the center of the image. In Figure~\ref{fig:exp2_syn}, we show the observed image with occlusions and our result. Figure~\ref{fig:exp2_con} contains the plots for the recovery error and the objective function values versus the iteration number in the Algorithm~\ref{alg:TIHT}.
One can see that our algorithm can achieve $10^{-4}$ for the recovery accuracy within 150 iterations while the other two comparing methods yield large errors $\geq 10^{-2}$. Notice that both CTI and EXT are based on exploiting the local image patch similarity while potentially skipping or putting less weight on preserving global repetitive patterns. They may struggle or fail when applied to fill large holes. In contrast, our algorithm prioritizes global similarity using the low tubal-rank data representation and performs better for this task. When the ground truth data has a relatively large tubal rank, Algorithm~\ref{alg:TIHT} will still converge to a decent result but very slowly. In the worst scenario when $r=\min\{n_1,n_2\}$, it behaves similar to gradient descent and $\cH_r$ takes no effect.

For the natural image test, we download a facade image, ``101\_rectified\_cropped'', from \url{https://people.ee.ethz.ch/~daid/FacadeSyn/} and then crop it to a color image of size $200\times 200\times 3$. Note that this image has tubal rank 3. In Figure~\ref{fig:exp2_nat}, we show the observed image with an occluded box of size $80\times 80$ in the center and our recovered result. Their recovery error and objective function value plots are shown in Figure~\ref{fig:exp2_con}, which show a swamp effect \cite{qi2021triple}. To achieve an RE of $10^{-4}$, more than 300 iterations are desired. This implies that the increase of tubal rank would request more iterations to achieve the same accuracy.

\setlength{\tabcolsep}{2pt}
\begin{figure}[t]
\centering
\begin{tabular}{cccc}
\includegraphics[width=.24\textwidth]{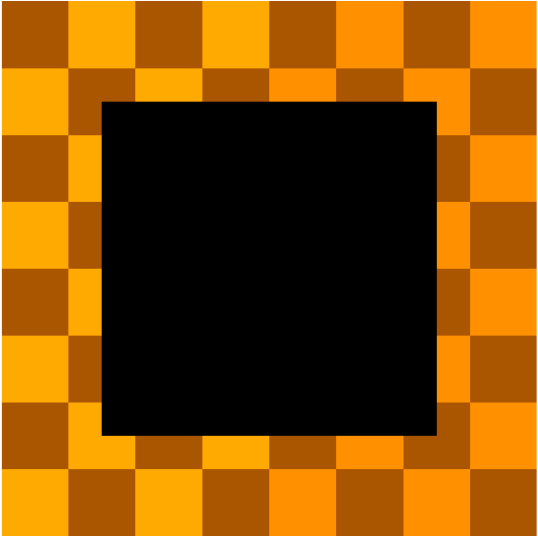}&
\includegraphics[width=.24\textwidth]{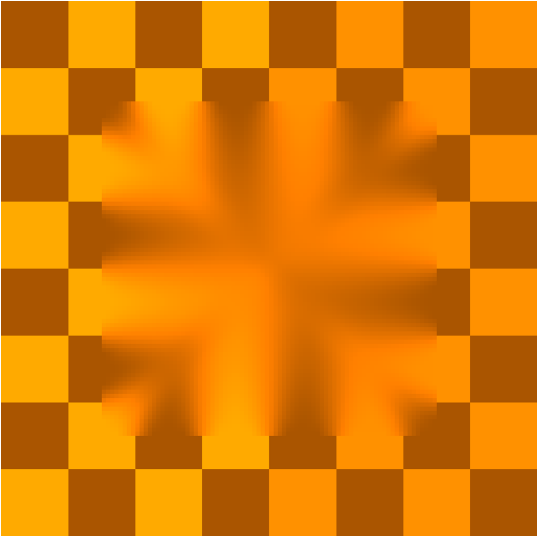}&
\includegraphics[width=.24\textwidth]{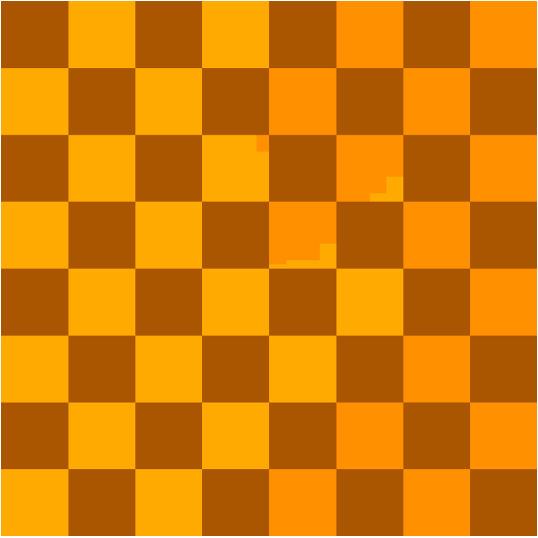}&
\includegraphics[width=.24\textwidth]{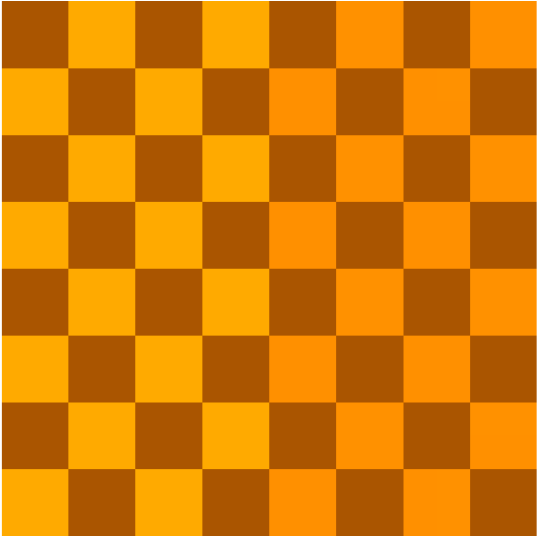}\\
(a) Observed image & (b) CTI &
(c) EXT  & (d) Our result\\
& RE=0.1945 & RE=0.0385 & RE=$1.29\times10^{-8}$
\end{tabular}
\vspace{-10pt}
\caption{Recovery of a checkerboard image with tubal rank 2 via Algorithm ~\ref{alg:TIHT}.}\label{fig:exp2_syn}
\end{figure}

\begin{figure}[t]
\centering
\begin{tabular}{cccc}
\includegraphics[width=.24\textwidth]{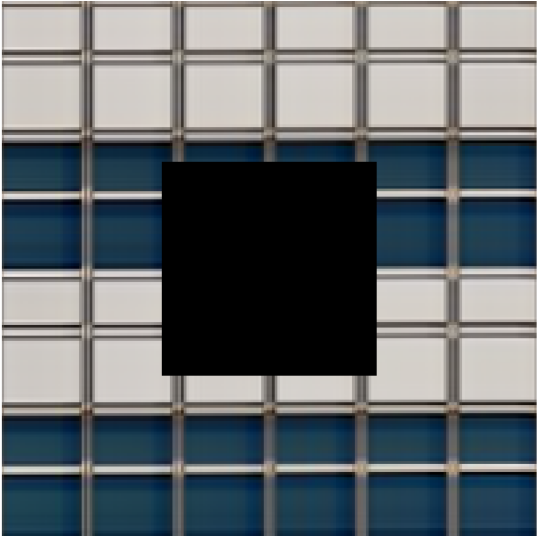}&
\includegraphics[width=.24\textwidth]{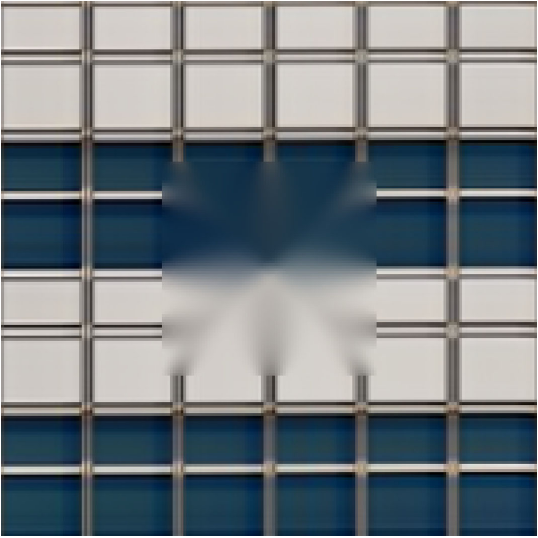}&
\includegraphics[width=.24\textwidth]{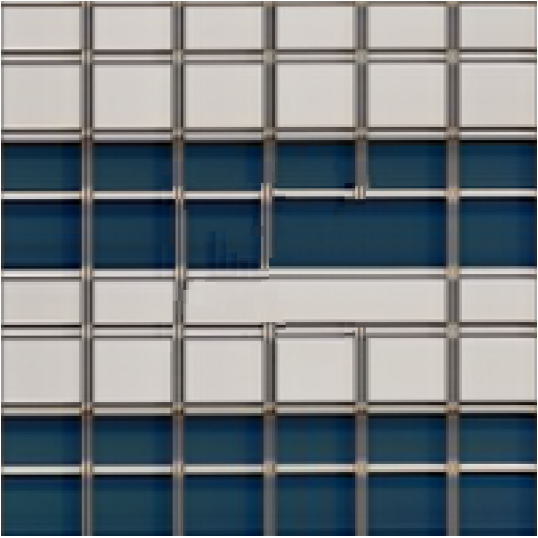}&
\includegraphics[width=.24\textwidth]{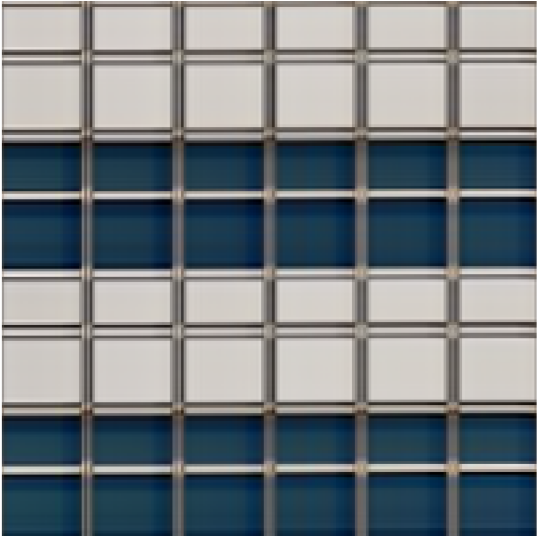}\\
(a) Observed image & (b) CTI &
(c) EXT & (d) Our result\\
& RE$=0.1222$& RE$=0.0658$ & RE$=3.89\times10^{-12}$
\end{tabular}
\vspace{-10pt}
\caption{Recovery of a facade image with tubal rank 3 via Algorithm ~\ref{alg:TIHT}.}\label{fig:exp2_nat}
\end{figure}

\begin{figure}[t]
\centering
\setlength{\tabcolsep}{4pt}
\begin{tabular}{cc}
\includegraphics[width=.45\textwidth]{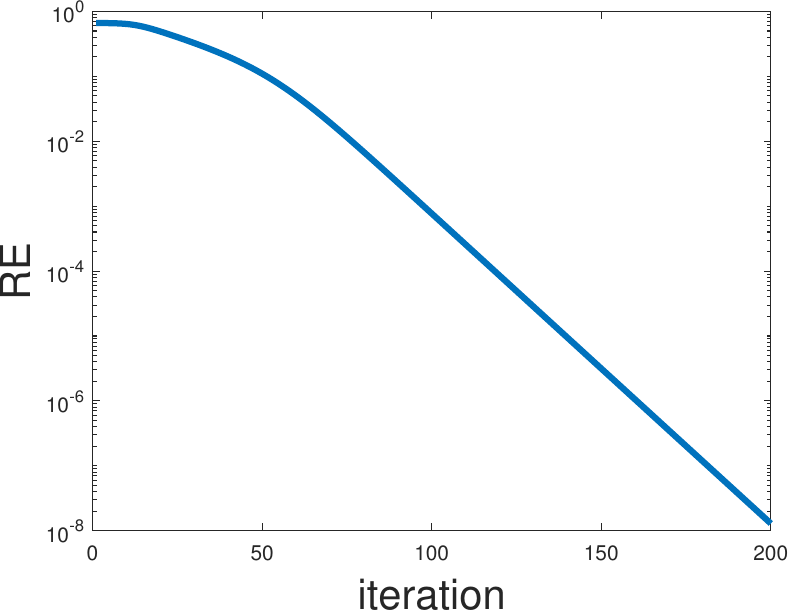}&
\includegraphics[width=.45\textwidth]{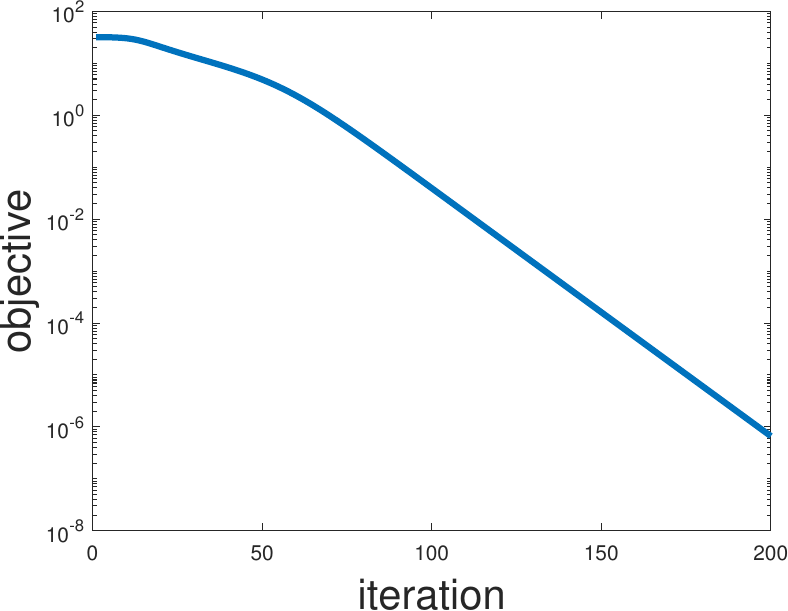}\\[-2pt]
\multicolumn{2}{c}{Convergence for the Checkerboard Image Test}\\[8pt]
\includegraphics[width=.45\textwidth]{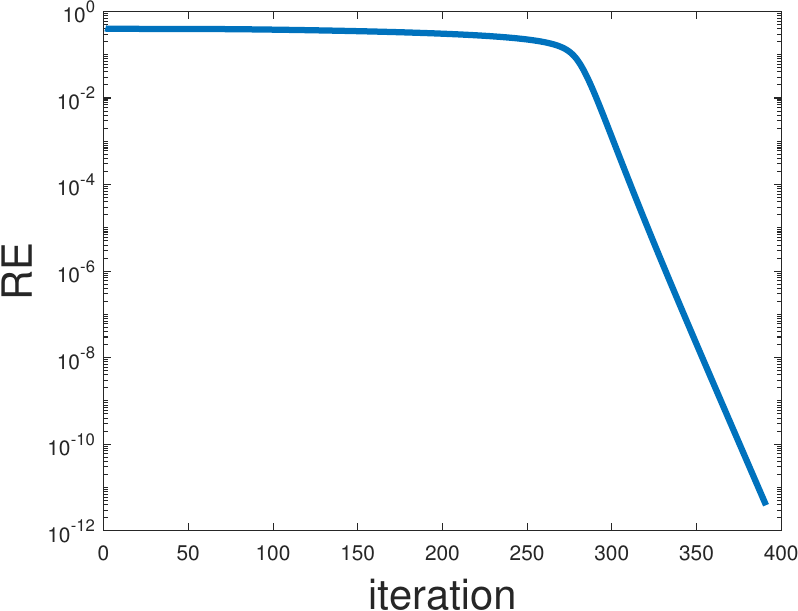}&
\includegraphics[width=.45\textwidth]{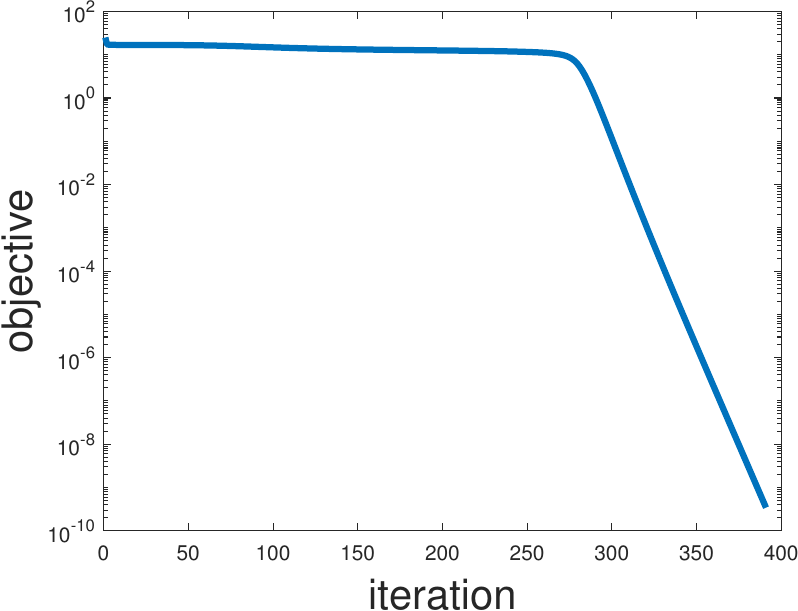}\\
 \multicolumn{2}{c}{Convergence for the Facade Image Test}
\end{tabular}
\vspace{-10pt}
\caption{Convergence of our method in color image inpainting.}\label{fig:exp2_con}
\end{figure}

\section{Conclusion}\label{sec:con}
The tensor has generalized the concept of the matrix but with more sophisticated features and computational challenges. Many tensor decompositions including CP, Tucker and t-SVD decompositions, have been developed to help analyze and manipulate large-scale data sets. Due to the computational efficiency and simple interpretation, t-SVD has recently attracted a lot of research attention, especially in the imaging field. In this work, we develop the iterative singular tube hard thresholding algorithm, which uses the t-SVD, together with its stochastic and batched stochastic versions. Tubal rank-restricted strong convexity and strong smoothness yield the convergence of the proposed algorithms.

\section*{CRediT Author Contributions}
Grotheer: validation, review \& editing, funding acquisition; Li: validation, review \& editing, revision, funding acquisition; Ma: investigation, formal analysis, review \& editing,  visualization, revision, funding acquisition; Needell: conceptualization, review \& editing, revision, project administration, supervision, funding acquisition; Qin: conceptualization, methodology, software, formal analysis, data curation, visualization, investigation, original draft writing, review \& editing, revision, funding acquisition.

\section*{Acknowledgments}
This material is based upon work supported by the National Security Agency under Grant No. H98230-19-1-0119, The Lyda Hill Foundation, The McGovern Foundation, and Microsoft Research, while the authors were in residence at the Mathematical Sciences Research Institute in Berkeley, California, during the summer of 2019. In addition, Grotheer was supported by the Goucher College Summer Research grant, Needell was funded by NSF CAREER DMS \#2011140 and NSF DMS \#2108479, and Qin is supported by NSF DMS \#1941197.

\bibliographystyle{unsrt}
\bibliography{ref}

\medskip
Received xxxx 20xx; revised xxxx 20xx; early access xxxx 20xx.
\medskip

\end{document}